\documentclass[11pt]{amsart}
\pdfoutput=1

\usepackage{amssymb}
\usepackage{mathtools}
\usepackage{microtype}
\usepackage[margin=1.25in,marginparwidth=1in]{geometry}

\usepackage[numbers]{natbib}
\usepackage{enumitem}
\usepackage{multirow}

\usepackage{aliascnt}
\usepackage{hyperref}
\hypersetup{
  pdftitle={Constraint-preserving hybrid finite element methods for
    Maxwell's equations},
  pdfauthor={Yakov Berchenko-Kogan and Ari Stern},
  pdfsubject={MSC 2010: 78M10, 65M60, 65N30},
  bookmarksopen=false,
}

\theoremstyle{plain}
\newtheorem{theorem}{Theorem}[section]
\newaliascnt{lemma}{theorem}

\aliascntresetthe{lemma}
\newaliascnt{proposition}{theorem}
\newtheorem{proposition}[proposition]{Proposition}
\aliascntresetthe{proposition}
\newaliascnt{corollary}{theorem}
\newtheorem{corollary}[corollary]{Corollary}
\aliascntresetthe{corollary}

\theoremstyle{definition}
\newaliascnt{definition}{theorem}

\aliascntresetthe{definition}

\theoremstyle{remark}
\newaliascnt{remark}{theorem}
\newtheorem{remark}[remark]{Remark}
\aliascntresetthe{remark}
\newaliascnt{example}{theorem}

\aliascntresetthe{example}

\begin{document}
\title[Constraint-preserving hybrid methods]{Constraint-preserving
  hybrid finite element methods for Maxwell's equations}

\author{Yakov Berchenko-Kogan}
\author{Ari Stern}
\address{Department of Mathematics and Statistics,
  Washington University in St.~Louis}
\email{yasha@wustl.edu}
\email{stern@wustl.edu}

\begin{abstract}
  Maxwell's equations describe the evolution of electromagnetic
  fields, together with constraints on the divergence of the magnetic
  and electric flux densities. These constraints correspond to
  fundamental physical laws: the nonexistence of magnetic monopoles
  and the conservation of charge, respectively.  However, one or both
  of these constraints may be violated when one applies a finite
  element method to discretize in space. This is a well-known and
  longstanding problem in computational electromagnetics.
    
  We use domain decomposition to construct a family of primal hybrid
  finite element methods for Maxwell's equations, where the Lagrange
  multipliers are shown to correspond to a numerical trace of the
  magnetic field and a numerical flux of the electric flux
  density. Expressing the charge-conservation constraint in terms of
  this numerical flux, we show that both constraints are strongly
  preserved. As a special case, these methods include a hybridized
  version of N\'ed\'elec's method, implying that it preserves the
  constraints more strongly than previously recognized. These
  constraint-preserving properties are illustrated using numerical
  experiments in both the time domain and frequency
  domain. Additionally, we observe a superconvergence phenomenon,
  where hybrid post-processing yields an improved estimate of the
  magnetic field.
\end{abstract}


\maketitle

\section{Introduction}

Maxwell's equations consist of two vector evolution equations,
together with two scalar constraint equations,
$ \operatorname{div} B = 0 $ and $ \operatorname{div} {D} = \rho $,
where $B$ is magnetic flux density, $D$ is electric flux density, and
$\rho$ is charge density. These constraints are automatically
preserved by the evolution, so given initial conditions satisfying the
constraints, one can simply evolve forward in time without needing to
``enforce'' the constraints in any way.

However, if one applies a finite element method in space, then the
semidiscretized evolution equations no longer necessarily preserve
these constraints, at least not strongly. \citet{Nedelec1980} showed
that, if one uses curl-conforming edge elements for the electric field
$E$ and divergence-conforming face elements for $B$, then the
semidiscretized equations preserve $ \operatorname{div} B = 0 $
strongly. On the other hand, $ \operatorname{div} {D} = \rho $ holds
only in the Galerkin sense (i.e., when both sides are integrated
against certain continuous, piecewise-polynomial test
functions). Observing this, \citet{ChWi2006} commented that strong
preservation of both divergence constraints ``appears to be necessary
for many applications in electromagnetics,'' and \citet{HoPeSc2004}
call this ``one of the main difficulties in the numerical solution of
Maxwell's equations.'' For this reason, alternative approaches have
been developed that enforce the constraints strongly---for instance,
using Lagrange multipliers \citep{AsDeHeRaSe1993,ChDuZo2000}---instead
of attempting to preserve them automatically but weakly, as
N\'ed\'elec's method does. In cases where $ \rho = 0 $, another idea
is to use divergence-free elements to construct nonconforming methods
\citep{BrLiSu2007,BrLiSu2009} or discontinuous Galerkin methods
\citep{CoLiSh2004,HoPeSc2004,BrLiSu2008}.

In this paper, we attack the problem of constraint preservation from a
different perspective. We perform domain decomposition of the
Lagrangian (i.e., primal) variational principle for Maxwell's
equations, in terms of the vector potential $A$ and scalar potential
$\varphi$, using Lagrange multipliers $ \widehat{ H } $ and
$ \widehat{ {D} } $ to enforce inter-element continuity and boundary
conditions. These Lagrange multipliers are shown to correspond to
boundary traces of the magnetic field $H$ and electric flux density
$D$. After using gauge symmetry to fix $ \varphi = 0 $, we show that
the evolution of $ ( A, \widehat{ H } ) $ automatically preserves the
constraints $ \operatorname{div} B = 0 $ and
$ \operatorname{div} \widehat{ {D} } = \rho $. Finally, we
semidiscretize this domain-decomposed variational principle, obtaining
primal hybrid finite element methods that preserve this formulation of
the constraints in a strong sense. As a special case, we give a
hybridized formulation of N\'ed\'elec's method, implying that it
preserves the constraints in a stronger sense than previously
recognized.

The paper is organized as follows:
\begin{itemize}
\item In \autoref{sec:maxwell}, we review Maxwell's equations, the
  Lagrangian variational principle, and semidiscretization using edge
  elements.

\item In \autoref{sec:dd}, we domain decompose the Lagrangian
  variational principle, relate solutions to the classical
  (non-domain-decomposed) formulation of Maxwell's equations, and
  study the domain-decomposed version of the constraints and their
  preservation.

\item In \autoref{sec:hybrid}, we consider primal hybrid finite
  element methods for semidiscretizing the domain-decomposed evolution
  equations, showing that constraints are preserved in a strong sense.

\item Finally, in \autoref{sec:numerical} we conduct numerical
  experiments demonstrating the behavior of the hybridized N\'ed\'elec
  method. In addition to the constraints being preserved to machine
  precision, these results illustrate a superconvergence phenomenon
  for the post-processed magnetic field $ \widehat{ H } _h $, similar
  to that observed for other hybridized mixed methods
  (cf.~\citet{ArBr1985,BrDoMa1985}).
\end{itemize}

\section{Maxwell's equations}
\label{sec:maxwell}

\subsection{Maxwell's equations}
\label{sec:maxwell_intro}

We begin by reviewing the classical formulation of Maxwell's
equations, first in terms of the electric and magnetic fields and flux
densities, and then in terms of the vector and scalar potentials. We
postpone the discussion of regularity until the introduction of the
weak formulation, in \autoref{sec:weak}; for the moment, everything
may be assumed to be smooth.

\subsubsection{Standard formulation}

In their most familiar form, Maxwell's equations consist of the vector
evolution equations,
\begin{subequations}
  \label{eqn:maxwell_evolution}
  \begin{align}
    \dot{ B } &= - \operatorname{curl} E, \label{eqn:maxwell_faraday}\\
    \dot{ {D} } + J &= \operatorname{curl} H  \label{eqn:maxwell_ampere},
  \end{align}
\end{subequations}
together with the scalar constraint equations,
\begin{subequations}
  \label{eqn:maxwell_constraints}
  \begin{align}
    \operatorname{div} B &= 0, \label{eqn:maxwell_divB}\\
    \operatorname{div} {D} &= \rho . \label{eqn:maxwell_divD}
  \end{align}
\end{subequations}
Here, $E$ and $H$ denote the electric field and magnetic field,
$ {D} = \epsilon E $ and $ B = \mu H $ denote the electric flux
density and magnetic flux density, $\epsilon$ and $\mu$ are the
electric permittivity and magnetic permeability tensors, and $J$ and
$\rho$ are current density and charge density, respectively. We use
the ``dot'' notation $ \dot{ u } \coloneqq \partial _t u $ to denote
partial differentiation with respect to time.

The evolution equations \eqref{eqn:maxwell_evolution} automatically
preserve the constraints \eqref{eqn:maxwell_constraints}. Indeed,
taking the divergence of \eqref{eqn:maxwell_faraday} implies
$ \operatorname{div} \dot{ B } = 0 $, so \eqref{eqn:maxwell_divB} is
preserved. Similarly, taking the divergence of
\eqref{eqn:maxwell_ampere} implies
$ \operatorname{div} \dot{ {D} } + \operatorname{div} J = 0 $, so
\eqref{eqn:maxwell_divD} is preserved if and only if $J$ and $\rho$
satisfy $ \dot{ \rho } + \operatorname{div} J = 0 $, which is the law
of conservation of charge. We refer to \eqref{eqn:maxwell_divD} as the
\emph{charge-conservation constraint}, since it is equivalent to this
condition.

\subsubsection{Formulation in terms of potentials}

Alternatively, Maxwell's equations may be expressed in terms of a
vector field $A$, called the \emph{vector potential}, and a scalar
field $\varphi$, called the \emph{scalar potential}. Given $A$ and
$\varphi$, we define the electric field and magnetic flux density by
\begin{equation*}
  E \coloneqq - ( \dot{ A } + \operatorname{grad} \varphi ) , \qquad B \coloneqq \operatorname{curl} A .
\end{equation*}
Note that \eqref{eqn:maxwell_faraday} and \eqref{eqn:maxwell_divB} are
automatically satisfied, so we may restrict our attention entirely to
the single evolution equation \eqref{eqn:maxwell_ampere}, which we
have already seen preserves \eqref{eqn:maxwell_divD}.

However, Maxwell's equations do not uniquely determine the evolution
of $ ( A, \varphi ) $. Observe that if $ \xi $ is any time-dependent
scalar field, then the transformation
$ ( A, \varphi ) \mapsto ( A + \operatorname{grad} \xi, \varphi -
\dot{ \xi } ) $ leaves $ E $, $B$, $D$, $H$ unchanged. Such
transformations are called \emph{gauge transformations}, and the
invariance of Maxwell's equations under gauge transformations is
called \emph{gauge symmetry}. In particular, any solution
$ ( A, \varphi ) $ may be transformed into one of the form
$ ( A + \operatorname{grad} \xi, 0 ) $ by taking $\xi$ to be a
solution of $ \dot{ \xi } = \varphi $. Therefore, we may restrict our
attention to solutions with $ \varphi = 0 $.

\begin{remark}
  \label{rmk:gauge}
  This procedure of restricting to particular solutions, which are
  related to a general solution by some gauge transformation, is
  called \emph{gauge fixing}. The choice $ \varphi = 0 $, called
  \emph{temporal gauge}, is the most convenient for our purposes, but
  there are other choices as well. Note that there is still some
  remaining gauge symmetry, even after performing temporal gauge
  fixing: we may transform $ A \mapsto A + \operatorname{grad} \xi $
  for any $\xi$ constant in time.
\end{remark}

After temporal gauge fixing, we can write \eqref{eqn:maxwell_ampere}
as either a first-order system in $A$, $D$,
\begin{equation*}
  \dot{ A } = - \epsilon ^{-1} {D} , \qquad \dot{ {D} } + J = \operatorname{curl} ( \mu ^{-1} \operatorname{curl} A ) ,
\end{equation*}
or as a second-order equation in $A$ alone,
\begin{equation*}
  - \partial _t ( \epsilon \dot{ A } ) + J = \operatorname{curl} ( \mu ^{-1} \operatorname{curl} A ) .
\end{equation*}
In the special case where $\epsilon$ and $\mu$ are simply positive
constants with $ \epsilon \mu = 1 $ (as in vacuum, with units chosen
so that the speed of light is $1$) and $ J = 0 $, the latter equation
just becomes
\begin{equation*}
  \ddot{ A } + \operatorname{curl} \operatorname{curl} A = 0 .
\end{equation*}
Taking the Fourier transform with respect to time (the so-called
\emph{frequency domain} or \emph{time-harmonic} approach), this latter
equation transforms into the eigenvalue problem for the
$ \operatorname{curl} \operatorname{curl} $ operator.

\subsection{Weak formulation}
\label{sec:weak}

We next discuss the weak formulation of Maxwell's equations, first
using a Lagrangian variational principle in terms of the potentials
$A$ and $\varphi$, and then fixing the temporal gauge $ \varphi = 0 $
to arrive at a weak formulation in terms of $A$ alone.

\subsubsection{Function spaces and regularity}
\label{sec:regularity}
Let $ \Omega \subset \mathbb{R}^3 $ be a bounded Lipschitz domain, and
define the function spaces
\begin{align*}
  H ^1 (\Omega) &\coloneqq \bigl\{ u \in L ^2 (\Omega) : \operatorname{grad} u \in L ^2 ( \Omega , \mathbb{R}^3  ) \bigr\} ,\\
  H ( \operatorname{curl} ; \Omega ) &\coloneqq \bigl\{ u \in L ^2 ( \Omega , \mathbb{R}^3  ) : \operatorname{curl} u \in L ^2 ( \Omega , \mathbb{R}^3  ) \bigr\} ,\\
  H ( \operatorname{div} ; \Omega ) &\coloneqq \bigl\{ u \in L ^2 ( \Omega ; \mathbb{R}^3  ) : \operatorname{div} u \in L ^2 (\Omega) \bigr\} .
\end{align*}
We also define the following subspaces, with boundary conditions imposed:
\begin{align*}
  \mathring{ H } ^1 (\Omega) &\coloneqq \bigl\{ u \in H ^1 (\Omega) : u \rvert _{ \partial \Omega } = 0 \bigr\} ,\\
  \mathring{ H } ( \operatorname{curl} ; \Omega ) &\coloneqq \bigl\{ u \in H ( \operatorname{curl} ; \Omega ) : u \times \mathbf{n} \rvert _{ \partial \Omega } = 0 \bigr\} ,\\
  \mathring{ H } ( \operatorname{div} ; \Omega ) &\coloneqq \bigl\{ u \in H ( \operatorname{div}; \Omega ) : u \cdot \mathbf{n} \rvert _{ \partial \Omega } = 0 \bigr\} .
\end{align*}
Here, $ \mathbf{n} \rvert _{ \partial \Omega } $ denotes the outer
unit normal to $ \partial \Omega $, and restrictions to
$ \partial \Omega $ are interpreted in the trace sense.

Let $ A \colon t \mapsto A (t) $ be a $ C ^1 $ curve in
$ \mathring{ H } ( \operatorname{curl} ; \Omega ) $ and
$ \varphi \colon t \mapsto \varphi (t) $ be a $ C ^0 $ curve in
$ \mathring{ H } ^1 (\Omega) $. It follows that $E$ is a $ C ^0 $
curve in $ \mathring{ H } ( \operatorname{curl} ; \Omega ) $, that
$ B $ is a $ C ^1 $ curve in
$ \mathring{ H } ( \operatorname{div}; \Omega ) $, and that
\eqref{eqn:maxwell_faraday} and \eqref{eqn:maxwell_divB} hold strongly
in $ L ^2 $. We also assume that both
$ \epsilon = \epsilon _{ i j } ( x, t ) $ and
$ \mu = \mu _{ i j } ( x, t ) $ are $ L ^\infty $, symmetric, and
uniformly elliptic. In particular, this implies that $ {D} $ and $H$
are both $ C ^0 $ curves in $ L ^2 ( \Omega , \mathbb{R}^3 )
$. Henceforth, we restrict our attention to $ ( A, \varphi ) $ such
that $ {D} $ is in fact a $ C ^1 $ curve in
$ L ^2 ( \Omega , \mathbb{R}^3 ) $.

Finally, let the current density $J$ be a given $ C ^0 $ curve in
$ H ( \operatorname{div} ; \Omega ) $ and the charge density $\rho$ be
a given $ C ^1 $ curve in $ L ^2 (\Omega) $, satisfying the charge
conservation condition $ \dot{ \rho } + \operatorname{div} J = 0 $.

\subsubsection{The Lagrangian and Euler--Lagrange equations}

For $ ( A, \varphi ) $ as above, define the Lagrangian
\begin{equation*}
  L ( A, \varphi , \dot{ A } , \dot{ \varphi } ) \coloneqq \int _\Omega \biggl( \frac{1}{2} E \cdot {D} - \frac{1}{2} B \cdot H + A \cdot J - \varphi \rho \biggr) .
\end{equation*}
The Euler--Lagrange equations are
\begin{subequations}
  \label{eqn:weak_maxwell}
\begin{alignat}{2}
  \int _\Omega \bigl( A ^\prime \cdot ( \dot{ {D} } + J ) - \operatorname{curl} A ^\prime \cdot H \bigr) &= 0 , &\qquad \forall A ^\prime &\in \mathring{ H } ( \operatorname{curl} ; \Omega ) ,\label{eqn:weak_ampere}\\
  \int _\Omega ( \operatorname{grad} \varphi ^\prime \cdot {D} +
  \varphi ^\prime \rho ) &= 0, & \forall \varphi ^\prime &\in
  \mathring{ H } ^1 (\Omega) ,\label{eqn:weak_divD}
\end{alignat}
\end{subequations}
which are weak expressions of \eqref{eqn:maxwell_ampere} and
\eqref{eqn:maxwell_divD}, respectively.

These Euler--Lagrange equations imply that solutions have additional
regularity properties. Since
$ \operatorname{curl} H = \dot{ {D} } + J $ is $ C ^0 $ in $ L ^2 $,
we have that $ H $ is $ C ^0 $ in
$ H ( \operatorname{curl} ; \Omega ) $. Likewise, since
$ \operatorname{div} {D} = \rho $ is $ C ^1 $ in $ L ^2 $, we have
that $ {D} $ is $ C ^1 $ in $ H ( \operatorname{div}; \Omega )
$. Hence, solutions to this weak problem are in fact strong solutions
of Maxwell's equations.

\begin{remark}
  When $\epsilon$ and $\mu$ are constant in time, the electric and
  magnetic fields have precisely the same regularity assumed by
  \citet[eqs. (7)--(8)]{Monk1993}, namely: $E$ is $ C ^1 $ in
  $ L ^2 ( \Omega , \mathbb{R}^3 ) $ and $ C ^0 $ in
  $ \mathring{ H } ( \operatorname{curl} ; \Omega ) $, while $H$ is
  $ C ^1 $ in $ L ^2 ( \Omega , \mathbb{R}^3 ) $ and $ C ^0 $ in
  $ H ( \operatorname{curl}; \Omega ) $.
\end{remark}

As in \autoref{sec:maxwell_intro}, this formulation is symmetric with
respect to gauge transformations
$ ( A, \varphi ) \mapsto ( A + \operatorname{grad} \xi, \varphi -
\dot{ \xi } ) $, where $\xi$ is now an arbitrary $ C ^1 $ curve in
$ \mathring{ H } ^1 (\Omega) $. Fixing the temporal gauge
$ \varphi = 0 $, the Lagrangian becomes
\begin{equation*}
  L ( A, \dot{ A } ) = \int _\Omega \biggl( \frac{1}{2} E \cdot {D} - \frac{1}{2} B \cdot H + A \cdot J \biggr),
\end{equation*} 
and the Euler--Lagrange equations are just
\eqref{eqn:weak_ampere}. This again implies that $ H $ is $ C ^0 $ in
$ H ( \operatorname{curl}; \Omega ) $, so \eqref{eqn:maxwell_ampere}
holds strongly.  By the same argument as in
\autoref{sec:maxwell_intro}, this automatically preserves the
charge-conservation constraint.

\begin{remark}
  Preservation of the charge-conservation constraint may also be seen
  as a consequence of the remaining gauge symmetry
  $ A \mapsto A + \operatorname{grad} \xi $, mentioned in
  \autoref{rmk:gauge}, where $ \xi \in \mathring{ H } ^1 (\Omega) $ is
  constant in time. This is a particular instance of Noether's
  theorem, which relates symmetries to conservation laws.  See
  \citet[Section 1.6]{MaRa1999} for an account of the $ J = 0 $ case,
  as well as the discussion in \citet{ChWi2006}.
\end{remark}

\subsection{Galerkin semidiscretization using N\'ed\'elec elements}
\label{sec:nedelec}

The use of finite elements in computational electromagnetics is a
broad topic with a long history, and we do not attempt to give a full
account here. We refer the reader to the texts by \citet{Monk2003} and
\citet{Jin2014}, as well as the excellent survey article by
\citet{Hiptmair2002}, which relates these methods to the more recent
theory of finite element spaces of differential forms. In this
section, we briefly review the semidiscretization of Maxwell's
equations using the elements of \citet{Nedelec1980,Nedelec1986}, an
approach that was subsequently analyzed in a series of papers by
\citet{Monk1991,Monk1992,Monk1993}.

Galerkin semidiscretization of the variational problem
\eqref{eqn:weak_ampere} restricts the trial and test functions to some
finite-dimensional subspace
$ V _h ^1 \subset \mathring{ H } ( \operatorname{curl} ; \Omega ) $,
resulting in a finite-dimensional system of ODEs. That is, we seek a
$ C ^1 $ curve $ A _h \colon t \mapsto A _h (t) \in V _h ^1 $ such
that
\begin{equation}
  \label{eqn:galerkin_ampere}
  \int _\Omega \bigl( A _h ^\prime \cdot ( \dot{ {D} } _h + J ) - \operatorname{curl} A _h ^\prime \cdot H _h \bigr)  = 0, \qquad \forall A _h ^\prime \in V _h ^1 ,
\end{equation}
where $ E _h \coloneqq - \dot{ A } _h $,
$ B _h \coloneqq \operatorname{curl} A _h $,
$ {D} _h \coloneqq \epsilon E _h $, and
$ H _h \coloneqq \mu ^{-1} B _h $. The discrete versions of
\eqref{eqn:maxwell_faraday} and \eqref{eqn:maxwell_divB},
\begin{equation*}
  \dot{ B } _h = - \operatorname{curl} E _h , \qquad \operatorname{div} B _h = 0 ,
\end{equation*}
follow immediately. In fact, both hold strongly in $ L ^2 $, by the
same argument as in \autoref{sec:regularity}, since
$ E _h \in V _h ^1 \subset \mathring{ H } ( \operatorname{curl};
\Omega ) $ and
$ B _h \in \operatorname{curl} V _h ^1 \subset \mathring{ H } (
\operatorname{div} ; \Omega ) $. On the other hand, we \emph{cannot}
conclude that $ {D} _h $ is in $ H ( \operatorname{div}; \Omega ) $,
nor that $ H _h $ is in $ H ( \operatorname{curl} ; \Omega ) $, since
\eqref{eqn:galerkin_ampere} only holds for test functions in
$ V _h ^1 $ and not all of
$ \mathring{ H } ( \operatorname{curl} ; \Omega ) $.

Consequently, the charge-conservation constraint
\eqref{eqn:maxwell_divD} is only preserved in the following, much
weaker sense. Let $ V _h ^0 \subset \mathring{ H } ^1 (\Omega) $ be a
finite-dimensional subspace such that
$ \operatorname{grad} V _h ^0 \subset V _h ^1 $. Then, for all
$ \xi _h \in V _h ^0 $, taking
$ A ^\prime _h = \operatorname{grad} \xi _h $ in
\eqref{eqn:galerkin_ampere} and applying
$ \dot{ \rho } + \operatorname{div} J = 0 $ gives
\begin{equation*}
  \int _\Omega ( \operatorname{grad} \xi _h \cdot \dot{ {D} } _h + \xi _h \dot{ \rho } ) = 0 .
\end{equation*} 
Hence, if the initial conditions satisfy
$ \int _\Omega ( \operatorname{grad} \xi _h \cdot {D} _h + \xi _h \rho
) = 0 $, for all $ \xi _h \in V _h ^0 $, then this is preserved by the
flow of \eqref{eqn:galerkin_ampere}.

In particular, suppose now that $ \Omega $ is polyhedral, and that
$ \mathcal{T} _h $ is a triangulation of $ \Omega $ by $3$-simplices
(i.e., tetrahedra) $ K \in \mathcal{T} _h $. We may take $ V _h ^0 $
to be the space of continuous degree-$r$ piecewise polynomials on
$ \mathcal{T} _h $ vanishing on $ \partial \Omega $, corresponding to
standard Lagrange finite elements. For $ V _h ^1 $, we may take either
degree-$r$ \citeauthor{Nedelec1980} edge elements of the first kind
\citep{Nedelec1980} or degree-$ ( r -1 ) $ \citeauthor{Nedelec1986}
edge elements of the second kind \citep{Nedelec1986} with vanishing
degrees of freedom on $ \partial \Omega $. These are spaces of
piecewise-polynomial vector fields in $ \mathbb{R}^3 $ with tangential
(but not necessarily normal) continuity between neighboring
simplices. These choices ensure that
$ \operatorname{grad} V _h ^0 \subset V _h ^1 $, so the weak
charge-conservation argument above holds.

Note, however, that
$ \int _\Omega ( \operatorname{grad} \xi _h \cdot {D} _h + \xi _h \rho
) = 0 $ only says that $ \operatorname{div} {D} _h = \rho $ holds in
an ``averaged'' sense, since (unlike in the infinite-dimensional case)
nonzero $ \xi _h \in V _h ^0 $ cannot be taken to have arbitrarily
small support. We cannot even conclude that the constraint holds in
the sense that
$ \int _{ \partial K } {D} _h \cdot \mathbf{n} = \int _K \rho $, since
the indicator function $ \mathbf{1} _K $ is discontinuous and
therefore not an admissible test function. (\citet{ChWi2006} give a
compactness argument for why this weak form of the constraint ``might
be just as good'' as the strong form, in the limit as
$ h \rightarrow 0 $; see also \citet{Christiansen2005}.) This
motivates our proposed hybrid approach, based on domain decomposition,
for which piecewise-constants \emph{are} admissible test functions.

\begin{remark}
  The method above describes the evolution of $ A _h \in V _h ^1
  $. Equivalently, one may evolve $ E _h \in V _h ^1 $ and
  $ B _h \in \operatorname{curl} V _h ^1 \subset V _h ^2 \subset
  \mathring{ H } ( \operatorname{div}; \Omega ) $ by augmenting
  \eqref{eqn:galerkin_ampere} with
  $ \dot{ B } _h = - \operatorname{curl} E _h $. This is the original
  approach described by \citet{Nedelec1980}, where $ V _h ^2 $ is
  given by face elements on $ \mathcal{T} _h $.
\end{remark}

\section{Domain decomposition}
\label{sec:dd}

In this section, we introduce an alternative variational formulation
for Maxwell's equations, based on domain decomposition. Specifically,
we decompose the problem on $ \Omega $ into a collection of problems
on $ K \in \mathcal{T} _h $, weakly enforcing internal continuity and
external boundary conditions using Lagrange multipliers. This is
similar in spirit to the standard approach to domain decomposition for
Poisson's equation, cf.~\citet{BrFo1991}. We show that the Lagrange
multipliers enforcing these conditions on $A$ and $\varphi$ correspond
to the traces of $H$ and $D$, respectively, and we show that the
latter satisfies an appropriate version of the charge-conservation
constraint.

\subsection{Function spaces}

We begin by introducing the following discontinuous function spaces,
which are larger than the spaces used in the previous variational
formulation:
\begin{align*}
  H ^1 ( \mathcal{T} _h ) &\coloneqq \bigl\{ u \in L ^2 (\Omega) : u \rvert _K \in H ^1 (K), \text{ for all } K \in \mathcal{T} _h \bigr\} ,\\
  H ( \operatorname{curl}; \mathcal{T} _h ) &\coloneqq \bigl\{ u \in L ^2 (\Omega, \mathbb{R}^3 ) : u \rvert _K \in H (\operatorname{curl}; K), \text{ for all } K \in \mathcal{T} _h \bigr\} ,\\
  H ( \operatorname{div} ; \mathcal{T} _h ) &\coloneqq \bigl\{ u \in L ^2 (\Omega, \mathbb{R}^3 ) : u \rvert _K \in H (\operatorname{div}; K), \text{ for all } K \in \mathcal{T} _h \bigr\} .
\end{align*}
\citet[Proposition III.1.1]{BrFo1991} show that
\begin{equation*}
  \mathring{ H } ^1 (\Omega) = \bigl\{ u \in H ^1 ( \mathcal{T} _h ) : \textstyle\sum _{ K \in \mathcal{T} _h } \int _{\partial K} u \lambda \cdot \mathbf{n} = 0 , \text{ for all } \lambda \in H ( \operatorname{div}; \Omega ) \bigr\} .
\end{equation*} 
That is, $ \mathring{ H } ^1 ( \Omega ) $ is the subspace of
$ H ^1 ( \mathcal{T} _h ) $ where internal continuity and external
boundary conditions are enforced by Lagrange multipliers
$ \lambda \in H ( \operatorname{div} ; \Omega ) $. Likewise, \citep[Proposition III.1.2]{BrFo1991} shows that
\begin{equation*}
  H ( \operatorname{div} ; \Omega ) = \bigl\{ u \in H ( \operatorname{div}; \mathcal{T} _h ) : \textstyle\sum _{ K \in \mathcal{T} _h } \int _{\partial K} u \lambda \cdot \mathbf{n} = 0 , \text{ for all } \lambda \in \mathring{ H } ^1  ( \Omega ) \bigr\} .
\end{equation*} 
Using a similar argument, we now prove the corresponding result for
the $ H ( \operatorname{curl} ) $ spaces.

\begin{proposition}
  \label{prop:hcurl}
  $ \mathring{ H } ( \operatorname{curl} ; \Omega ) = \bigl\{ u \in H
  ( \operatorname{curl}; \mathcal{T} _h ) : \sum _{ K \in \mathcal{T}
    _h } \int _{ \partial K } (u \times \lambda ) \cdot \mathbf{n} = 0
  , \text{ for all } \lambda \in H ( \operatorname{curl}; \Omega )
  \bigr\} $.
\end{proposition}

\begin{proof}
  If
  $ u \in \mathring{ H } ( \operatorname{curl} ; \Omega ) \subset H (
  \operatorname{curl}; \mathcal{T} _h ) $, then for any
  $ \lambda \in H ( \operatorname{curl} ; \Omega ) $, we have
  \begin{align*}
    \sum _{ K \in \mathcal{T} _h } \int _{ \partial K } ( u \times \lambda ) \cdot \mathbf{n}
    &= \sum _{ K \in \mathcal{T} _h } \int _K ( \operatorname{curl} u \cdot \lambda - u \cdot \operatorname{curl} \lambda ) \\
    &= \int _\Omega ( \operatorname{curl} u \cdot \lambda - u \cdot \operatorname{curl} \lambda ) \\
    &= \int _{ \partial \Omega } ( u \times \lambda ) \cdot \mathbf{n}\\
    &= 0,
  \end{align*}
  so the forward inclusion $ ( \subset ) $ holds. To get the reverse
  inclusion $ ( \supset ) $, suppose that
  $ u \in H ( \operatorname{curl} ; \mathcal{T} _h ) $ satisfies the
  condition above, and let
  $ \lambda \in C ^\infty _c ( \Omega , \mathbb{R}^3 ) $. Then,
  integrating by parts, we have
  \begin{align*}
    \biggl\lvert \int _\Omega u \cdot \operatorname{curl} \lambda \biggr\rvert
    &= \Biggl\lvert \sum _{ K \in \mathcal{T} _h } \int _K \operatorname{curl} u \cdot \lambda - \sum _{ K \in \mathcal{T} _h } \int _{ \partial K } (u \times \lambda ) \cdot \mathbf{n} \Biggr\rvert \\
    &= \Biggl\lvert \sum _{ K \in \mathcal{T} _h } \int _K \operatorname{curl} u \cdot \lambda \Biggr\rvert\\
    &\leq \Biggl( \sum _{ K \in \mathcal{T} _h } \lVert  \operatorname{curl} u \rVert ^2 _{ L ^2 ( K , \mathbb{R}^3  ) } \Biggr) ^{ 1/2 } \lVert \lambda \rVert _{ L ^2 ( \Omega , \mathbb{R}^3  ) },
  \end{align*}
  where the last line uses the triangle and Cauchy--Schwarz
  inequalities. It follows that
  $ \operatorname{curl} u \in L ^2 ( \Omega , \mathbb{R}^3 ) $, so
  $ u \in H ( \operatorname{curl} ; \Omega ) $. This implies that
  $ \int _{\partial \Omega} ( u \times \lambda ) \cdot \mathbf{n} =
  \sum _{ K \in \mathcal{T} _h } \int _{ \partial K } (u \times
  \lambda ) \cdot \mathbf{n} = 0 $ for all
  $ \lambda \in H ( \operatorname{curl} ; \Omega ) $. Hence,
  $ u \times \mathbf{n} \rvert _{ \partial \Omega } = 0 $ in the trace
  sense, which completes the proof.
\end{proof}

\begin{remark}
  A variant of this result is stated in \citet[Proposition
  2.1.3]{BoBrFo2013}, where $ \lambda $ is taken to be in
  $ H ^1 ( \Omega , \mathbb{R}^3 ) $ rather than
  $ H ( \operatorname{curl}; \Omega ) $.  However, the version given
  here is more natural for the purposes of the hybrid methods
  discussed in \autoref{sec:hybrid}.
\end{remark}

\subsection{Domain decomposition of the Lagrangian variational principle}

We now introduce a new Lagrangian for Maxwell's equations, which
allows the potentials to live in the discontinuous function spaces
defined in the previous section, enforcing continuity and boundary
conditions using Lagrange multipliers.

Let $ A (t) \in H ( \operatorname{curl} ; \mathcal{T} _h ) $ and
$ \varphi (t) \in H ^1 ( \mathcal{T} _h ) $, and introduce the
Lagrange multipliers
$ \widehat{ {H} } (t) \in H ( \operatorname{curl} ; \Omega ) $ and
$ \widehat{ {D} } (t) \in H ( \operatorname{div}; \Omega ) $. We adopt
the notation, often seen in the literature on discontinuous Galerkin
and hybrid methods, of placing hats over variables that act like weak
traces/fluxes. As before, suppose that $t \mapsto A (t) $ is $ C ^1 $
and that $ t \mapsto \varphi (t) $ is $ C ^0 $, such that
$ t \mapsto {D} (t) \in L ^2 ( \Omega , \mathbb{R}^3 ) $ is $ C ^1
$. Furthermore, suppose that $ t \mapsto \widehat{ H } (t) $ and
$ t \mapsto \widehat{ {D} } (t) $ are both $ C ^0 $.  Define the
Lagrangian
\begin{multline*}
  L ( A, \varphi , \widehat{ H }, \widehat{ {D} } , \dot{ A } , \dot{ \varphi },  \dot{ \widehat{ H } }, \dot{ \widehat{ {D} } }  ) = \sum _{ K \in \mathcal{T} _h } \Biggl[ \int _K \biggl( \frac{1}{2} E \cdot {D} - \frac{1}{2} B \cdot H + A \cdot J - \varphi \rho \biggr) \\
  + \int _{ \partial K } ( A \times \widehat{ H } + \varphi \widehat{
    {D} } ) \cdot \mathbf{n} \Biggr].
\end{multline*}
The Euler--Lagrange equations are then
\begin{subequations}
  \label{eqn:dd_maxwell}
  \begin{alignat}{2}
    \int _K \bigl( A ^\prime \cdot ( \dot{ {D} } + J ) - \operatorname{curl} A ^\prime \cdot H \bigr) + \int _{ \partial K } ( A ^\prime \times \widehat{ H } ) \cdot \mathbf{n} &= 0, \qquad & \forall A ^\prime &\in H ( \operatorname{curl} ; K ) ,\label{eqn:dd_ampere}\\
    \int _K ( \operatorname{grad} \varphi ^\prime \cdot {D} + \varphi ^\prime \rho ) - \int _{ \partial K } \varphi ^\prime \widehat{ {D} } \cdot \mathbf{n} &= 0, &\forall \varphi ^\prime &\in H ^1 (K) , \label{eqn:dd_divD}\\
    \sum _{ K \in \mathcal{T} _h } \int _{ \partial K } ( A \times \widehat{ H } ^\prime ) \cdot \mathbf{n} &= 0 , &\forall \widehat{ H } ^\prime &\in H ( \operatorname{curl}; \Omega ) ,\label{eqn:dd_continuity_A}\\
    \sum _{ K \in \mathcal{T} _h } \int _{ \partial K } \varphi
    \widehat{ {D} } ^\prime \cdot \mathbf{n} &= 0, & \forall \widehat{ {D} } ^\prime &\in H ( \operatorname{div}; \Omega ) ,\label{eqn:dd_continuity_phi}
  \end{alignat}
\end{subequations}
where \eqref{eqn:dd_ampere} and \eqref{eqn:dd_divD} hold for all
$K \in \mathcal{T} _h $. We now relate this to the classical
variational form of Maxwell's equations, stated in
\eqref{eqn:weak_maxwell}.

\begin{proposition}
  $ ( A, \varphi , \widehat{ {H} } , \widehat{ {D} } ) $ is a solution
  to \eqref{eqn:dd_maxwell} if and only if $ ( A, \varphi ) $ is a
  solution to \eqref{eqn:weak_maxwell} with
  $ \widehat{ H } \times \mathbf{n} \rvert _{ \partial K } = H \times
  \mathbf{n} \rvert _{ \partial K } $ and
  $ \widehat{ {D} } \cdot \mathbf{n} \rvert _{ \partial K } = {D}
  \cdot \mathbf{n} \rvert _{ \partial K } $. In particular, if
  $ ( A, \varphi ) $ is a solution to \eqref{eqn:weak_maxwell}, then
  $ ( A, \varphi , H , {D} ) $ is a solution to
  \eqref{eqn:dd_maxwell}.
\end{proposition}

\begin{proof}
  Suppose $ ( A, \varphi , \widehat{ {H} } , \widehat{ {D} } ) $ is a
  solution to \eqref{eqn:dd_maxwell}. By \autoref{prop:hcurl},
  \eqref{eqn:dd_continuity_A} implies
  $ A (t) \in \mathring{ H } ( \operatorname{curl} ; \Omega ) $, so
  taking
  $ A ^\prime \in \mathring{ H } ( \operatorname{curl} ; \Omega ) $
  and summing \eqref{eqn:dd_ampere} over $ K \in \mathcal{T} _h $, the
  integrals over $ \partial K $ cancel, yielding
  \eqref{eqn:weak_ampere}. As previously stated,
  \eqref{eqn:weak_ampere} implies
  $ \operatorname{curl} H = \dot{ {D} } + J $, so substituting this
  into \eqref{eqn:dd_ampere} gives
  \begin{equation*}
    \int _{ \partial K } ( A ^\prime \times \widehat{ H } ) \cdot \mathbf{n} = \int _K ( \operatorname{curl} A ^\prime \cdot H - A ^\prime \cdot \operatorname{curl} H ) = \int _{ \partial K } ( A ^\prime \times  H ) \cdot \mathbf{n} , \qquad \forall A ^\prime \in H ( \operatorname{curl} ; K ) ,
  \end{equation*} 
  so
  $ \widehat{ H } \times \mathbf{n} \rvert _{ \partial K } = H \times
  \mathbf{n} \rvert _{ \partial K } $. Similarly,
  \eqref{eqn:dd_continuity_phi} implies
  $ \varphi (t) \in \mathring{ H } ^1 (\Omega) $, so taking
  $ \varphi ^\prime \in \mathring{ H } ^1 (\Omega) $ and summing
  \eqref{eqn:dd_divD} over $ K \in \mathcal{T} _h $ yields
  \eqref{eqn:weak_divD}. This implies
  $ \operatorname{div} {D} = \rho $, and substituting into
  \eqref{eqn:dd_divD} gives
  $ \widehat{ {D} } \cdot \mathbf{n} \rvert _{ \partial K } = {D}
  \cdot \mathbf{n} \rvert _{ \partial K } $.

  Conversely, suppose $ ( A, \varphi ) $ is a solution to
  \eqref{eqn:weak_maxwell}. Since
  $ A (t) \in \mathring{ H } ( \operatorname{curl} ; \Omega ) $ and
  $ \varphi (t) \in \mathring{ H } ^1 (\Omega) $, it follows that
  \eqref{eqn:dd_continuity_A} and \eqref{eqn:dd_continuity_phi}
  hold. Furthermore, \eqref{eqn:weak_maxwell} implies that
  $ \dot{ {D} } + J = \operatorname{curl} H $ and
  $ \operatorname{div} {D} = \rho $, so \eqref{eqn:dd_ampere} and
  \eqref{eqn:dd_divD} hold with
  $ \widehat{ {H} } \times \mathbf{n} \rvert _{ \partial K } = H
  \times \mathbf{n} \rvert _{ \partial K } $ and
  $ \widehat{ {D} } \cdot \mathbf{n} \rvert _{ \partial K } = {D}
  \cdot \mathbf{n} \rvert _{ \partial K } $. In particular, we could
  take $ \widehat{ H } = H $ and $ \widehat{ {D} } = {D} $.
\end{proof}

\begin{remark}
  \label{rmk:divDhat}
  Note that, in addition to \eqref{eqn:dd_divD} implying that
  $ \operatorname{div} {D} = \rho $, we also see by taking
  $ \varphi ^\prime = \mathbf{1} _K $ that $ \widehat{ {D} } $
  satisfies the conservation law
  $ \int _{ \partial K } \widehat{ {D} } \cdot \mathbf{n} = \int _K
  \rho $, for all $ K \in \mathcal{T} _h $.
\end{remark}

\subsection{Temporal gauge fixing and the charge-conservation
  constraint}

As in \autoref{sec:weak}, if
$ ( A, \varphi , \widehat{ H } , \widehat{ {D} } ) $ is a solution to
\eqref{eqn:dd_maxwell}, then so is
$ ( A + \operatorname{grad} \xi, \varphi - \dot{ \xi } , \widehat{ H }
, \widehat{ {D} } ) $ for any $ C ^1 $ curve
$ t \mapsto \xi (t) \in \mathring{ H } ^1 (\Omega) $. Therefore, we
perform temporal gauge fixing by taking $ \varphi = 0 $. This yields
the gauge-fixed Lagrangian
\begin{equation*}
  L ( A, \widehat{ H }, \dot{ A } , \dot{ \widehat{ H } } ) = \sum _{ K \in \mathcal{T} _h } \Biggl[ \int _K \biggl( \frac{1}{2} E \cdot {D} - \frac{1}{2} B \cdot H + A \cdot J  \biggr) 
  + \int _{ \partial K }  (A \times \widehat{ H } ) \cdot \mathbf{n} \Biggr],
\end{equation*} 
whose Euler--Lagrange equations are simply \eqref{eqn:dd_ampere} and
\eqref{eqn:dd_continuity_A}. Of course, \eqref{eqn:dd_continuity_phi}
is satisfied trivially, since $ \varphi = 0 $. The next result shows
that the charge-conservation constraint \eqref{eqn:dd_divD} is
automatically preserved, for an appropriately-defined
$ \widehat{ {D} } $.

\begin{proposition}
  \label{prop:dd_maxwell_conservation}
  Let $ ( A, \widehat{ H } ) $ be a solution to \eqref{eqn:dd_ampere}
  and \eqref{eqn:dd_continuity_A}. Suppose initial values for
  $ {D}, \widehat{ {D} } $ satisfy \eqref{eqn:dd_divD}, and let
  $ \widehat{ {D} } $ be the solution to
  $ \dot{ \widehat{ {D} } } + J = \operatorname{curl} \widehat{ H }
  $. Then $ ( A, 0, \widehat{ H } , \widehat{ {D} } ) $ is a solution
  to \eqref{eqn:dd_maxwell}.
\end{proposition}

\begin{proof}
  As we have already mentioned, $ \varphi = 0 $ trivially satisfies
  \eqref{eqn:dd_continuity_phi}, so it suffices to show that
  \eqref{eqn:dd_divD} holds.  Let $ \varphi ^\prime \in H ^1 (K) $ be
  arbitrary. Taking
  $ A ^\prime = \operatorname{grad} \varphi ^\prime $ in
  \eqref{eqn:dd_ampere} and integrating by parts gives
  \begin{align*}
    0 &= \int _K \operatorname{grad}  \varphi  ^\prime \cdot ( \dot{ {D} } + J ) + \int _{ \partial K } (\operatorname{grad} \varphi ^\prime \times \widehat{ H } ) \cdot \mathbf{n} \\
      &= \int _K ( \operatorname{grad} \varphi ^\prime \cdot \dot{ {D} } - \varphi ^\prime \operatorname{div} J  ) + \int _{ \partial K } \varphi ^\prime ( J - \operatorname{curl} \widehat{ H }  ) \cdot \mathbf{n} \\
      &= \int _K ( \operatorname{grad} \varphi ^\prime \cdot \dot{ {D} } + \varphi ^\prime \dot{ \rho }  ) - \int _{ \partial K } \varphi ^\prime \dot{ \widehat{ {D} } }  \cdot \mathbf{n} ,
  \end{align*}
  so if \eqref{eqn:dd_divD} holds at the initial time, then it holds
  for all time.
\end{proof}

\begin{remark}
  As in \autoref{rmk:divDhat}, taking
  $ \varphi ^\prime = \mathbf{1} _K $ implies
  $ \int _{ \partial K } \widehat{ {D} } \cdot \mathbf{n} = \int _K
  \rho $. Furthermore, if the initial conditions also satisfy
  $ \operatorname{div} \widehat{ {D} } = \rho $, then we have
  $ \operatorname{div} \widehat{ {D} } = \rho $ for all time, since
  $ \operatorname{div} \dot{ \widehat{ {D} } } = \operatorname{div}
  \operatorname{curl} \widehat{ H } - \operatorname{div} J = 0 + \dot{
    \rho } $.  Finally, if $ \widehat{ H } = H $, and if the initial
  conditions for $ \widehat{ {D} } $ equal those for $ {D} $, then we
  recover $ \widehat{ {D} } = {D} $.
\end{remark}

Finally, we express this variational problem in the standard notation
used for mixed and hybrid finite element methods, in terms of a pair
of bilinear forms \citep[Chapter II]{BrFo1991}. We will make use of
this notation throughout the subsequent sections. Defining
\begin{align*}
  a &\colon H ( \operatorname{curl} ; \mathcal{T} _h ) \times H ( \operatorname{curl} ; \mathcal{T} _h ) \rightarrow \mathbb{R}  , & a ( A , A ^\prime  ) &\coloneqq \sum _{ K \in \mathcal{T} _h } \int _K \operatorname{curl} A ^\prime \cdot \mu ^{-1} \operatorname{curl} A  ,\\
  b &\colon H ( \operatorname{curl} ; \mathcal{T} _h ) \times H ( \operatorname{curl} ; \Omega ) \rightarrow \mathbb{R}  , & b ( A ^\prime , \widehat{ H } ) &\coloneqq - \sum _{ K \in \mathcal{T} _h } \int _{ \partial K } ( A ^\prime \times \widehat{ H } ) \cdot \mathbf{n} ,
\end{align*}
we seek
$ t \mapsto A (t) \in H ( \operatorname{curl} ; \mathcal{T} _h ) $ and
$ t \mapsto \widehat{ H } (t) \in H ( \operatorname{curl} ; \Omega ) $
such that
\begin{subequations}
  \label{eqn:dd_abstract}
\begin{alignat}{2}
  \langle \dot{ {D} } + J , A ^\prime \rangle &= a ( A , A ^\prime ) + b ( A ^\prime, \widehat{ H } ) , &\qquad \forall A ^\prime &\in H ( \operatorname{curl} ; \mathcal{T} _h ) ,\label{eqn:dd_abstract_a}\\
  0 &= b ( A, \widehat{ H } ^\prime ) , &\forall \widehat{ H } ^\prime &\in H ( \operatorname{curl} ; \Omega ) \label{eqn:dd_abstract_b},
\end{alignat}
\end{subequations}
where $ \langle \cdot , \cdot \rangle $ is the
$ L ^2 ( \Omega , \mathbb{R}^3 ) $ inner product. Defining the map
$ \mathcal{B} \colon H ( \operatorname{curl} ; \mathcal{T} _h )
\rightarrow H ( \operatorname{curl} ; \Omega ) ^\ast $,
$ A \mapsto b ( A, \cdot ) $, we see that \eqref{eqn:dd_abstract} is
equivalent to evolving $ A (t) \in \ker \mathcal{B} $ by
\begin{equation*}
  \langle \dot{ {D} } + J , A ^\prime \rangle = a ( A , A ^\prime )  , \qquad \forall A ^\prime \in \ker \mathcal{B} ,
\end{equation*} 
and subsequently solving for $ \widehat{ H } $ satisfying
\eqref{eqn:dd_abstract_a}. Since
$ \ker \mathcal{B} = \mathring{ H } ( \operatorname{curl} ; \Omega ) $
by \autoref{prop:hcurl}, it follows that $A$ solves the
non-domain-decomposed problem \eqref{eqn:weak_ampere}.

\section{Hybrid semidiscretization}
\label{sec:hybrid}

We now perform Galerkin semidiscretization of the domain-decomposed
variational problem with temporal gauge fixing, as introduced in the
previous section. This results in a hybrid method for Maxwell's
equations, where ``hybrid'' means that the Lagrange multipliers
$ \widehat{ H } _h $ and their test functions
$ \widehat{ H } _h ^\prime $ are both restricted to a subspace of
$ H ( \operatorname{curl}; \Omega ) $. We then show that a
suitably-defined $ \widehat{ {D} } _h $ satisfies the
charge-conservation constraint in a strong sense, as opposed to the
much weaker sense in which $ {D} _h $ was seen to satisfy this
constraint in \autoref{sec:nedelec}. Finally, we discuss how certain
choices of elements yield a hybridized version of N\'ed\'elec's
method, while others give nonconforming methods, and we remark on how
this framework also applies to the 2-{D} Maxwell equations.

\subsection{Semidiscretization of the variational problem}
\label{sec:hybrid_variational}
For each $ K \in \mathcal{T} _h $, let
$ V _h ^1 (K) \subset H ( \operatorname{curl}; K ) $ be a
finite-dimensional subspace, so 
$ V _h ^1 \coloneqq \prod _{ K \in \mathcal{T} _h } V _h ^1 (K)
\subset H ( \operatorname{curl} ; \mathcal{T} _h ) $, and let
$ \widehat{ V } _h ^1 \subset H ( \operatorname{curl} ; \Omega ) $. We
seek $ A _h \colon t \mapsto A _h (t) \in V _h ^1 $ and
$ \widehat{ H } _h \colon t \mapsto \widehat{ H } _h (t) \in \widehat{ V }
_h ^1 $ such that
\begin{subequations}
  \label{eqn:hybrid_maxwell}
\begin{alignat}{2}
  \int _K \bigl( A _h ^\prime \cdot ( \dot{ {D} } _h + J ) - \operatorname{curl} A _h ^\prime \cdot H _h \bigr) + \int _{ \partial K } ( A _h ^\prime \times \widehat{ H } _h ) \cdot \mathbf{n} &= 0, \qquad & \forall A _h ^\prime &\in V _h ^1 (K)  ,\label{eqn:hybrid_ampere}\\
  \sum _{ K \in \mathcal{T} _h } \int _{ \partial K } ( A _h \times
  \widehat{ H } _h ^\prime ) \cdot \mathbf{n} &= 0 , &\forall \widehat{ H
  } _h ^\prime &\in \widehat{ V } _h ^1 ,\label{eqn:hybrid_continuity_A}
\end{alignat}
\end{subequations}
where \eqref{eqn:hybrid_ampere} holds for all $K \in \mathcal{T} _h
$. These are the semidiscretized versions of \eqref{eqn:dd_ampere} and
\eqref{eqn:dd_continuity_A}.

\begin{remark}
  Since \eqref{eqn:hybrid_continuity_A} only holds for test functions
  in $ \widehat{ V } _h ^1 $, but not necessarily an arbitrary test
  function in $ H ( \operatorname{curl}; \Omega ) $, in general a
  solution will have
  $ A _h (t) \notin \mathring{ H } ( \operatorname{curl} ; \Omega )
  $. Hence, this method is generally not curl-conforming and is
  distinct from the conforming methods discussed in
  \autoref{sec:nedelec}.
\end{remark}

In terms of the bilinear forms $ a ( \cdot , \cdot ) $ and
$ b ( \cdot , \cdot ) $, this method may be written as
\begin{subequations}
  \label{eqn:hybrid_abstract}
\begin{alignat}{2}
  \langle \dot{ {D} } _h + J , A _h ^\prime \rangle &= a ( A _h , A _h  ^\prime ) + b ( A _h ^\prime, \widehat{ H } _h ) , &\qquad \forall A _h  ^\prime &\in V _h ^1  ,\label{eqn:hybrid_abstract_a}\\
  0 &= b ( A _h , \widehat{ H } _h ^\prime ) , &\forall \widehat{ H } _h ^\prime &\in \widehat{ V } _h ^1  \label{eqn:hybrid_abstract_b}.
\end{alignat}
\end{subequations}
Defining the operator
$ \mathcal{B} _h \colon V _h ^1 \rightarrow ( \widehat{ V } _h ^1 )
^\ast $,
$ A _h \mapsto b ( A _h , \cdot ) \rvert _{ \widehat{ V } _h ^1 } $,
we see that \eqref{eqn:hybrid_abstract} is equivalent to evolving
$ A _h (t) \in \ker \mathcal{B} _h $ by
\begin{equation}
  \label{eqn:hybrid_kerBh}
  \langle \dot{ {D} } _h + J , A _h ^\prime \rangle = a ( A _h , A _h  ^\prime ) , \qquad \forall A _h  ^\prime \in \ker \mathcal{B} _h ,
\end{equation}
and subsequently solving for $ \widehat{ H } _h $ satisfying
\eqref{eqn:hybrid_abstract_a}.

Since $ V _h ^1 $ is finite-dimensional, we may apply Banach's closed
range theorem to deduce that
$ \langle \dot{ {D} } _h + J , \cdot \rangle - a ( A _h , \cdot ) \in
(\ker \mathcal{B} _h) ^\perp $ is in the range of
$ \mathcal{B} _h ^\ast $, so a solution $ \widehat{ H } _h $ exists,
although generally not uniquely. A natural choice is to find the
solution $ \widehat{ H } _h $ minimizing
$ \lVert H _h - \widehat{ H } _h \rVert ^2 + \lVert \dot{ {D} } _h + J
- \operatorname{curl} \widehat{ H } _h \rVert ^2 $, which in a weak
sense minimizes the $ H ( \operatorname{curl} ; \Omega ) $ distance
between $ H _h $ and $ \widehat{ H } _h $. This
existence-without-uniqueness is typical of hybrid methods, and one may
formally resolve this by replacing $ \widehat{ V } _h ^1 $ by the
quotient space $ \widehat{ V } _h ^1 / \ker \mathcal{B} _h ^\ast $
(cf.~\citet[IV.1.3]{BrFo1991}). In practice, the evolution on
$ \ker \mathcal{B} _h $ specified by \eqref{eqn:hybrid_kerBh} is the
essence of the method, and solving for $ \widehat{ H } _h $ may be
seen as an optional post-processing step.

\subsection{Preservation of the charge-conservation constraint}
\label{sec:charge_conservation}
In order to discuss the charge-conservation constraint, we first
suppose that $ V _h ^0 (K) \subset H ^1 (K) $ are such that
$ \mathbf{1} _K \in V _h ^0 (K) $ and
$ \operatorname{grad} V _h ^0 (K) \subset V _h ^1 (K) $ for all
$ K \in \mathcal{T} _h $. We consider whether the following
discretization of \eqref{eqn:dd_divD} is preserved,
\begin{equation}
  \label{eqn:hybrid_divD}
  \int _K ( \operatorname{grad} \varphi _h ^\prime \cdot {D} _h + \varphi _h ^\prime \rho ) - \int _{ \partial K } \varphi _h ^\prime \widehat{ {D} } _h \cdot \mathbf{n} = 0, \qquad \forall \varphi ^\prime _h \in V _h ^0  (K) ,
\end{equation}
for
$ \widehat{ {D} } _h \colon t \mapsto \widehat{ {D} } _h (t) \in H (
\operatorname{div} ; \Omega ) $ suitably defined.

\begin{theorem}
  \label{thm:hybrid_maxwell_conservation}
  Let $ ( A _h , \widehat{ H } _h ) $ be a solution to
  \eqref{eqn:hybrid_maxwell}. Suppose initial values for
  $ {D} _h , \widehat{ {D} } _h $ satisfy \eqref{eqn:hybrid_divD}, and
  let $ \widehat{ {D} } _h $ be the solution to
  $ \dot{ \widehat{ {D} } } _h + J = \operatorname{curl} \widehat{ H }
  _h $. Then \eqref{eqn:hybrid_divD} holds for all time. In
  particular,
  $ \int _{ \partial K } \widehat{ {D} } _h \cdot \mathbf{n} = \int _K
  \rho $. Moreover, if
  $ \operatorname{div} \widehat{ {D} } _h = \rho $ holds at the
  initial time, then it holds for all time.
\end{theorem}

\begin{proof}
  The proof is essentially similar to that of
  \autoref{prop:dd_maxwell_conservation}. Given
  $ \varphi _h ^\prime \in V _h ^0 (K) $, taking
  $ A _h ^\prime = \operatorname{grad} \varphi _h ^\prime \in V _h ^1
  (K) $ in \eqref{eqn:hybrid_ampere} and integrating by parts,
  \begin{align*}
    0 &= \int _K \operatorname{grad} \varphi _h ^\prime \cdot ( \dot{ {D} } _h + J ) + \int _{ \partial K } ( \operatorname{grad} \varphi _h ^\prime \times \widehat{ H } _h ) \cdot \mathbf{n} \\
      &= \int _K ( \operatorname{grad} \varphi _h ^\prime \cdot \dot{ {D} } _h - \varphi _h ^\prime \operatorname{div} J ) + \int _{ \partial K }  \varphi _h ^\prime ( J - \operatorname{curl} \widehat{ H } _h ) \cdot \mathbf{n} \\
    &= \int _K ( \operatorname{grad} \varphi _h ^\prime \cdot \dot{ {D} } _h + \varphi _h ^\prime \dot{ \rho } ) - \int _{ \partial K } \varphi _h ^\prime \dot{ \widehat{ {D} } } _h \cdot \mathbf{n} ,
  \end{align*}
  so if \eqref{eqn:hybrid_divD} holds at the initial time, then it
  holds for all time. The conclusion that
  $ \int _{ \partial K } \widehat{ {D} } _h \cdot \mathbf{n} = \int _K
  \rho $ follows by taking $ \varphi _h ^\prime = \mathbf{1} _K $, and
  $ \operatorname{div} \dot{ \widehat{ {D} } } _h = \operatorname{div}
  \operatorname{curl} \widehat{ H } _h - \operatorname{div} J = 0 +
  \dot{ \rho } $ implies that if
  $ \operatorname{div} \widehat{ {D} } _h = \rho $ holds at the
  initial time, then it holds for all time.
\end{proof}

\begin{remark}
  Preservation of $ \operatorname{div} \widehat{ {D} } _h = \rho $ is
  immediate from
  $ \dot{ \widehat{ {D} } } _h + J = \operatorname{curl} \widehat{ H }
  _h $, without appealing to \eqref{eqn:hybrid_divD}.  However, it is
  only a meaningful statement about solutions to
  \eqref{eqn:hybrid_maxwell} when \eqref{eqn:hybrid_divD} holds.  By
  contrast, if $ \widehat{ {D} } _h $ were instead to satisfy
  $ \dot{ \widehat{ {D} } } _h + J = 0 $, then
  $ \operatorname{div} \widehat{ {D} } _h = \rho $ would still be
  preserved, but this would not say anything about the numerical
  solution $ ( A _h , \widehat{ H } _h ) $.
\end{remark}

The next result addresses the existence of initial conditions for
$ \widehat{ {D} } _h $ satisfying the hypotheses of the previous
theorem. Let
$ V _h ^0 \coloneqq \prod _{ K \in \mathcal{T} _h } V _h ^0 (K)
\subset H ^1 ( \mathcal{T} _h ) $.

\begin{proposition}
  \label{prop:initial_value_Dhat}
  Suppose that the initial value of $ {D} _h $ satisfies
  \begin{equation*}
    \sum _{ K \in \mathcal{T} _h } \int _K \operatorname{grad} \varphi _h ^\prime \cdot {D} _h + \int _\Omega \varphi _h ^\prime \rho = 0 , \qquad \forall \varphi _h ^\prime \in V _h ^0 \cap \mathring{ H } ^1 (\Omega) .
  \end{equation*}
  Then there exists an initial value for $ \widehat{ {D} } _h $ such
  that \eqref{eqn:hybrid_divD} holds for all $ K \in \mathcal{T} _h $
  and $ \operatorname{div} \widehat{ {D} } _h = \rho $.
\end{proposition}

\begin{proof}
  The first part of the argument is similar to the one we used for the
  existence of $ \widehat{ H } _h $.  Define the bilinear form
  \begin{equation*}
    \beta _h \colon V _h ^0 \times H ( \operatorname{div}; \Omega ) \rightarrow \mathbb{R}  , \qquad \beta _h ( \varphi _h ^\prime , \widehat{ {D} } _h ) \coloneqq \sum _{ K \in \mathcal{T} _h } \int _{ \partial K } \varphi _h ^\prime \widehat{ {D} } _h \cdot \mathbf{n} .
  \end{equation*} 
  Since $ V _h ^0 \cap \mathring{ H } ^1 (\Omega) $ is the kernel of
  $ \varphi _h ^\prime \mapsto \beta _h ( \varphi _h ^\prime , \cdot )
  $, the closed range theorem implies that
  $ \varphi _h ^\prime \mapsto \sum _{ K \in \mathcal{T} _h } \int _K
  \operatorname{grad} \varphi _h ^\prime \cdot {D} _h + \int _\Omega
  \varphi _h ^\prime \rho $ is in the range of
  $ \widehat{ {D} } _h \mapsto \beta _h ( \cdot , \widehat{ {D} } _h )
  $. Hence, there exists an initial value for $ \widehat{ {D} } _h $
  satisfying \eqref{eqn:hybrid_divD} for all $K \in \mathcal{T} _h $.

  Next, suppose $ \widehat{ {D} } _h $ satisfies
  \eqref{eqn:hybrid_divD} but not necessarily
  $ \operatorname{div} \widehat{ {D} } _h = \rho $. Then, on each
  $ K \in \mathcal{T} _h $, replace $ \widehat{ {D} } _h $ by
  $ \widehat{ {D} } _h + \operatorname{grad} u $, where $u$ is the
  solution to
  $ - \Delta u = \operatorname{div} \widehat{ {D} } _h - \rho $ with
  Neumann boundary conditions
  $ \operatorname{grad} u \cdot \mathbf{n} = 0 $ on $ \partial K
  $. This leaves the normal traces of $ \widehat{ {D} } _h $
  unchanged, so the result is still in
  $ H ( \operatorname{div}; \Omega ) $ and satisfies
  \eqref{eqn:hybrid_divD}, as desired.
\end{proof}

\begin{remark}
  The computation of $ \widehat{ {D} } _h $, like that of
  $ \widehat{ H } _h $, can be seen as an optional post-processing
  step after computing the solution $ A _h $ to
  \eqref{eqn:hybrid_kerBh}. The key point of
  \autoref{thm:hybrid_maxwell_conservation} is that the evolution of
  $ A _h $ is conservative, in the sense that it is consistent with a
  charge-conserving numerical flux $ \widehat{ {D} } _h $, whether or
  not one chooses to actually compute $ \widehat{ {D} } _h $.
\end{remark}

\subsection{Hybridization of N\'ed\'elec's method and nonconforming
  methods}
\label{sec:hybrid_nedelec}

As in \autoref{sec:nedelec}, let $ \Omega $ be polyhedral and
$ \mathcal{T} _h $ be a simplicial triangulation. Let $ V _h ^0 (K) $
be the space of degree-$r$ polynomials on $K$ and $ V _h ^1 (K) $ be
either degree-$r$ N\'ed\'elec edge elements of the first kind or
degree-$(r-1)$ N\'ed\'elec edge elements of the second kind on
$K$. Then $ V _h ^0 \subset H ^1 ( \mathcal{T} _h ) $ and
$ V _h ^1 \subset H ( \operatorname{curl} ; \mathcal{T} _h ) $
correspond to discontinuous Lagrange and N\'ed\'elec elements,
respectively. Note that discontinuous N\'ed\'elec elements of the
second kind are just discontinuous piecewise polynomial vector fields.

Now, taking
$ \widehat{ V } _h ^1 = H ( \operatorname{curl}; \Omega ) $, it
follows that
$ \ker \mathcal{B} _h = V _h ^1 \cap \ker \mathcal{B} \subset
\mathring{ H } ( \operatorname{curl} ; \Omega ) $, which corresponds
precisely to curl-conforming N\'ed\'elec elements with tangential
inter-element continuity and boundary conditions. It follows that
\eqref{eqn:hybrid_kerBh} agrees precisely with N\'ed\'elec's method
\eqref{eqn:galerkin_ampere}. In fact, it is not necessary to take
$ \widehat{ V } _h ^1 $ infinite-dimensional: it suffices to take a
large enough finite-dimensional subspace (e.g., N\'ed\'elec elements
of sufficiently high degree) such that \eqref{eqn:hybrid_continuity_A}
imposes all the inter-element continuity and boundary conditions on
degrees of freedom of $ V _h ^1 $. (Having $ \widehat{ V } _h ^1 $
infinite-dimensional is not a problem if one is only interested in
$ A _h $, but a finite-dimensional subspace is required if one wishes
to compute $ \widehat{ H } _h $.) From these observations, we obtain
the following corollary of \autoref{thm:hybrid_maxwell_conservation}
and \autoref{prop:initial_value_Dhat}

\begin{corollary}
  Given $ V _h ^0 $ and $ V _h ^1 $ as above, there exists
  $ \widehat{ V } _h ^1 $ such that solutions $ A _h $ to
  N\'ed\'elec's method \eqref{eqn:galerkin_ampere} are equivalent to
  solutions $ ( A _h , \widehat{ H } _h ) $ to the hybrid method
  \eqref{eqn:hybrid_maxwell}. Consequently, given a solution to
  N\'ed\'elec's method, there exists $ \widehat{ {D} } _h $ satisfying
  $ \dot{ \widehat{ {D} } } _h + J = \operatorname{curl} \widehat{ H }
  _h $, which preserves the charge-conservation constraints
  \eqref{eqn:hybrid_divD} and
  $ \operatorname{div} \widehat{ {D} } _h = \rho $.
\end{corollary}

In contrast, if $ \widehat{ V } _h ^1 $ is not sufficiently large, we
will have
$ \ker \mathcal{B} _h \not\subset \ker \mathcal{B} = \mathring{ H } (
\operatorname{curl}; \Omega ) $, so \eqref{eqn:hybrid_kerBh} is a
nonconforming finite element method for Maxwell's equations.

\subsection{Remarks on the two-dimensional case}
\label{sec:2D}

This framework may also be adapted to two-dimensional electromagnetics
with minor modifications.

For the non-domain-decomposed problem on
$ \Omega \subset \mathbb{R}^2 $, the potential
$A \in \mathring{ H } ( \operatorname{curl}; \Omega ) $ remains a
vector field, although $ \operatorname{curl} A \in L ^2 (\Omega) $
becomes a scalar field. Consequently, $E$ and $D$ remain vector fields
(and $\epsilon$ remains a tensor), while $B$ and $H$ become scalar
fields (and $\mu$ becomes scalar). The two-dimensional version of the
weak problem \eqref{eqn:weak_ampere} is nearly identical, except the
dot product $ \operatorname{curl} A ^\prime \cdot H $ is replaced by
the ordinary product $ ( \operatorname{curl} A ^\prime ) H $. For the
Galerkin semidiscretization discussed in \autoref{sec:nedelec}, one
simply replaces the N\'ed\'elec edge elements of the first and second
kind with Raviart--Thomas (RT) \citep{RaTh1977} and
Brezzi--Douglas--Marini (BDM) \citep{BrDoMa1985} edge elements,
respectively. These two-dimensional $ H ( \operatorname{curl} ) $
elements are just the RT and BDM
$ H ( \operatorname{div} ) $ elements rotated by 90 degrees, so that
tangential traces of the former correspond to normal traces of the
latter.

For domain decomposition, \autoref{prop:hcurl} is easily modified to
show that
\begin{equation*}
  \mathring{ H } ( \operatorname{curl} ; \Omega ) = \bigl\{ u \in H
  ( \operatorname{curl}; \mathcal{T} _h ) : \sum _{ K \in \mathcal{T}
    _h } \int _{ \partial K } u \lambda \times  \mathbf{n} = 0
  , \text{ for all } \lambda \in H ^1 ( \Omega ) \bigr\} .
\end{equation*}
Alternatively, this can be seen to follow from the corresponding
result for $ \mathring{ H } ( \operatorname{div}; \Omega ) $, where
the vector fields are rotated by 90 degrees. Hence, the domain
decomposed variational problem in temporal gauge is to find
$ t \mapsto A (t) \in H ( \operatorname{curl} ; \mathcal{T} _h ) $ and
$ t \mapsto \widehat{ H } (t) \in H ^1 ( \Omega ) $ such that
\begin{alignat*}{2}
  \int _K \bigl(  A ^\prime \cdot ( \dot{ {D} } + J ) - ( \operatorname{curl} A ^\prime ) H \bigr)  + \int _{ \partial K } A ^\prime \widehat{ H } \times \mathbf{n} &= 0 , \qquad & \forall A ^\prime & \in H ( \operatorname{curl} ; K ) ,\\
  \sum _{ K \in \mathcal{T} _h } \int _{ \partial K } A \widehat{ H }
  ^\prime \times \mathbf{n} &= 0 , \qquad & \forall \widehat{ H }
  ^\prime &\in H ^1 (\Omega) ,
\end{alignat*}
for all $ K \in \mathcal{T} _h $. Hybrid methods may then be obtained
by restricting this variational problem to subspaces
$ V _h ^1 = \prod _{ K \in \mathcal{T} _h } V _h ^1 (K) \subset H (
\operatorname{curl} ; \mathcal{T} _h ) $ and
$ \widehat{ V } _h ^0 \subset H ^1 (\Omega) $. As in
\autoref{sec:charge_conservation}, one obtains
$ \widehat{ {D} } _h (t) \in H ( \operatorname{div}; \Omega ) $ by
solving
$ \dot{ \widehat{ {D} } } _h + J = \operatorname{curl} \widehat{ H }
_h $ (where the curl of a scalar field is its gradient rotated by 90
degrees, i.e.,
$ v \cdot \operatorname{curl} \widehat{ H } _h \coloneqq v \times
\operatorname{grad} \widehat{ H } _h $ for $ v \in \mathbb{R}^2 $),
and the charge-conserving properties follow in the same manner.

For the finite element spaces, one may take $ V _h ^0 $ to be
discontinuous degree-$r$ Lagrange elements and $ V _h ^1 $ to be
discontinuous degree-$r$ RT edge elements or discontinuous
degree-$(r-1)$ BDM edge elements. (Note that discontinuous BDM
elements are just discontinuous piecewise polynomial vector fields.)
In this case, it is much easier to see which
$ \widehat{ V } _h ^0 \subset H ^1 (\Omega) $ yield conforming
methods, since each edge degree of freedom is either shared by exactly
two triangles or lies on the boundary. Both the degree-$r$ RT and
degree-$(r-1)$ BDM elements have $r$ degrees of freedom per edge,
which match up precisely with those for degree-$(r+1)$ Lagrange
elements. Hence, taking $ \widehat{ V } _h ^0 $ corresponding to
degree-$(r+1)$ or higher Lagrange elements yields a conforming
method. On the other hand, a straightforward counting argument shows
that degree-$r$ Lagrange elements have fewer than
$ r \times \text{\#edges} $ degrees of freedom on element boundaries
(unless $ \mathcal{T} _h $ consists of a single triangle). Since it is
impossible to enforce all of the inter-element and boundary conditions
in this case, the resulting method is nonconforming.

\section{Numerical examples}
\label{sec:numerical}

This section gives numerical illustrations for the simple test problem
\begin{equation}
  \label{eqn:curlcurl_time}
  \ddot{ A } + \operatorname{curl} \operatorname{curl} A = 0 ,
\end{equation}
which corresponds to the case where $\epsilon$ and $\mu$ are positive
constants with $ \epsilon \mu = 1 $ and $ J = 0 $, as discussed at the
end of \autoref{sec:maxwell_intro}. As before, $A$ is taken to have
vanishing tangential component on the boundary. Preservation of the
charge-conservation constraint is equivalent to the condition
$ \operatorname{div} \ddot{ A } = 0 $.

In the frequency domain, denoting angular frequency by $\omega$, time
differentiation becomes multiplication by $ i \omega $, so
\eqref{eqn:curlcurl_time} becomes the eigenvalue problem
\begin{equation}
  \label{eqn:curlcurl_frequency}
  \operatorname{curl} \operatorname{curl} A = \omega ^2 A .
\end{equation}
In this setting, preservation of the charge-conservation constraint
becomes $ \omega ^2 \operatorname{div} A = 0 $, i.e., eigenfunctions
with nonzero eigenvalue are divergence-free.

The examples below demonstrate the constraint-preserving properties of
the curl-conforming hybridized N\'ed\'elec method from
\autoref{sec:hybrid}, both in the time domain and in the frequency
domain. In the frequency domain, we also observe superconvergence of
$ \widehat{ H } _h \rightarrow H $. All finite element computations
were performed using FEniCS
\citep{LoMaWe2012,AlBlHaJoKeLoRiRiRoWe2015}. For the post-processing
step of computing $ \widehat{ H } _h $, whose solution is not unique,
we find the solution $ \widehat{ H } _h $ minimizing
$ \lVert H _h - \widehat{ H } _h \rVert ^2 + \lVert \dot{ {D} } _h + J
- \operatorname{curl} \widehat{ H } _h \rVert ^2 $, as previously
discussed in \autoref{sec:hybrid_variational}.

\subsection{Time domain}

Before turning our attention to the test problem
\eqref{eqn:curlcurl_time}, we first describe a discrete time-stepping
scheme for the general case of Maxwell's equations. After
semidiscretizing using the hybridized N\'ed\'elec method of
\autoref{sec:hybrid}, we discretize in time using the following
explicit ``leapfrog'' scheme:
\begin{itemize}

\item
  $ A _{ n + 1/2 } = A _n - \frac{1}{2} \Delta t \epsilon ^{-1} {D} _n
  $.

\item $ {D} _{ n + 1 } = {D} _n + \Delta t \dot{ {D} } _{ n + 1/2 } $,
  where $ \dot{D} _{ n + 1/2 } \in \ker \mathcal{B} _h $ is the solution to
  \begin{equation*}
    \langle \dot{ {D} } _{ n + 1/2 }  + J _{ n + 1/2 } , A _h ^\prime \rangle = a ( A _{n+1/2} , A _h  ^\prime ) , \qquad \forall A _h  ^\prime \in \ker \mathcal{B} _h.
  \end{equation*}

\item
  $ \widehat{ {D} } _{ n + 1 } = \widehat{ {D} } _n + \Delta t (
  \operatorname{curl} \widehat{ H } _{ n + 1/2 } - J _{ n + 1/2} ) $,
  where $ \widehat{ H } _{ n + 1/2 } $ is the solution to
  \begin{equation*}
    \langle \dot{ {D} } _{n+1/2} + J_{n+1/2} , A _h ^\prime \rangle = a ( A _{n+1/2} , A _h  ^\prime ) + b ( A _h ^\prime, \widehat{ H } _{n+1/2} ) , \qquad \forall A _h  ^\prime \in V _h ^1 ,
  \end{equation*}
  minimizing
  $ \lVert H _{n+1/2} - \widehat{ H } _{n+1/2} \rVert ^2 + \lVert
  \dot{ {D} } _{n+1/2} + J _{n+1/2} - \operatorname{curl} \widehat{ H
  } _{n+1/2} \rVert ^2 $.
  
\item
  $ A _{ n + 1 } = A _{n+1/2} - \frac{1}{2} \Delta t \epsilon ^{-1}
  {D} _{n+1} $.

\end{itemize}
Here, $ A _n $ denotes the approximation to $ A _h ( t _n ) $, where
$ t _n $ is the $n$th time step and $ \Delta t $ is the time step
size; similar notation is used for the other variables. This is
essentially the St\"ormer/Verlet method for the semidiscretized system
of ODEs \eqref{eqn:hybrid_kerBh}, augmented by a hybrid
post-processing step for $ \widehat{ H } _h $ and
$ \widehat{ {D} } _h $. Except for the hybrid post-processing step,
which is novel, such leapfrog schemes are widely used for both finite
element and finite difference time domain methods in computational
electromagnetics (see \citet{Yee1966} and \citet[Section
5]{Monk1991}). The St\"ormer/Verlet method also has particularly
desirable properties when applied to Lagrangian and Hamiltonian
dynamics (cf.~\citet{HaLuWa2003,HaLuWa2006}).

\begin{figure}
  \centering
  \includegraphics[scale=.8]{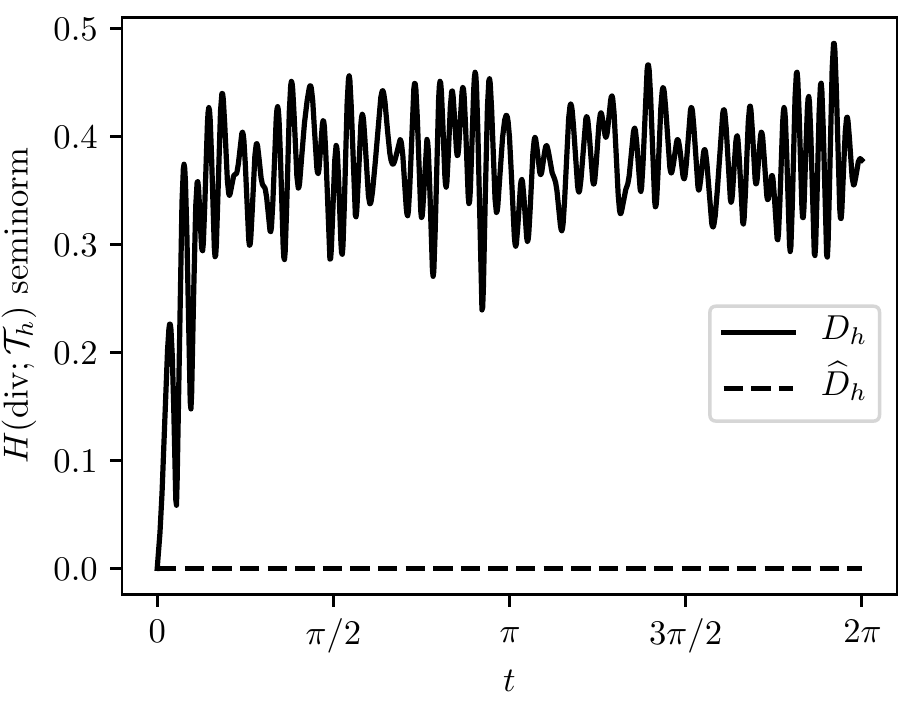} \hfill
  \includegraphics[scale=.8]{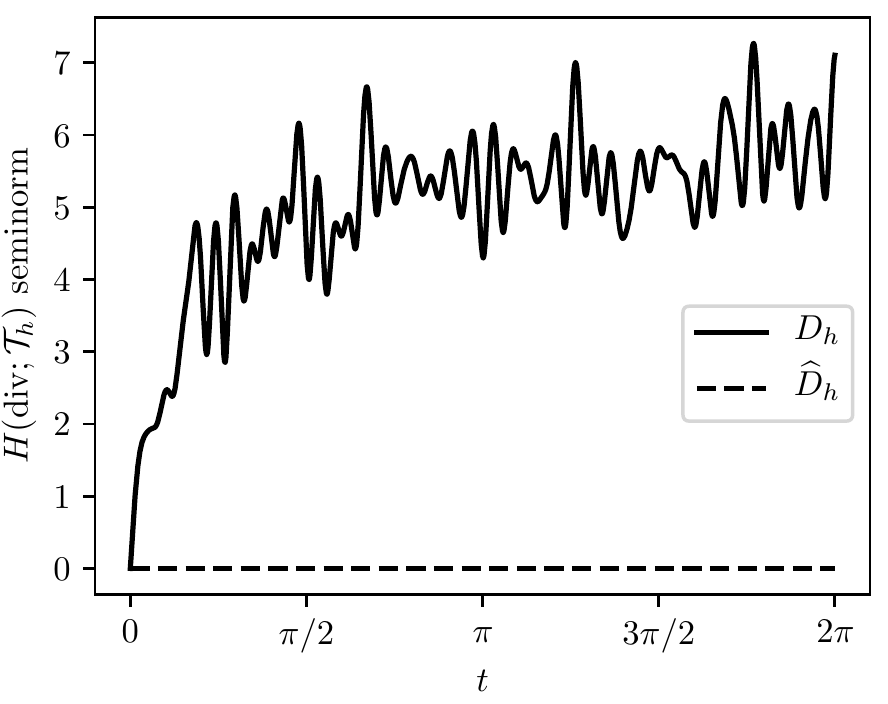}
  \caption{Charge conservation error, as measured by the
    $ H ( \operatorname{div} ; \mathcal{T} _h ) $ seminorm of
    $ {D} _h $ and $ \widehat{ {D} } _h $, over time, on the 2-{D} square
    $ \Omega = ( 0, \pi ) ^2 $ (left) and 3-{D} cube
    $ \Omega = ( 0, \pi ) ^3 $ (right). Although $ {D} _h $ drifts
    away from the constraint, $ \widehat{ {D} } _h $ preserves the
    constraint to machine precision.\label{fig:time}}
\end{figure}

\autoref{fig:time} shows the results of applying this method to the
test problem \eqref{eqn:curlcurl_time} on the 2-{D} square
$ \Omega = ( 0, \pi ) ^2 $ and 3-{D} cube $ \Omega = ( 0, \pi ) ^3 $,
taking $ \epsilon = \mu = 1 $. For both the 2-{D} and 3-{D} problems,
we simulate over $ t \in [ 0, 2 \pi ] $ for $ 1024 $ time steps of
size $ \Delta t = \pi / 512 $.

For the 2-{D} problem, the initial conditions are taken to be
$ {D} _0 = \widehat{ {D} } _0 = 0 $ and
\begin{equation*}
  A _0 ( x, y ) = \bigl( y ( \pi - y ) , x ( \pi - x ) \bigr) .
\end{equation*} 
A uniform triangular mesh is taken on a $ 16 \times 16 $ grid, with
$ 2 \cdot 16 ^2 = 512 $ cells. The space $ V _h ^1 $ consists of
discontinuous piecewise linear vector fields, while
$ \widehat{ V } _h ^0 $ consists of cubic Lagrange elements, so that
$ \ker \mathcal{B} _h \subset V _h ^1 $ are linear BDM edge elements,
as described in \autoref{sec:2D} with $ r = 2 $.

For the 3-{D} problem, the initial conditions are taken to be
$ {D} _0 = \widehat{ {D} } _0 = 0 $ and
\begin{equation*}
  A _0 ( x, y, z ) = \bigl( y ( \pi - y ) z ( \pi - z ) , z ( \pi - z ) x ( \pi - x ), x ( \pi - x ) y ( \pi - y ) \bigr) .
\end{equation*} 
A uniform tetrahedral mesh is taken on an $ 8 \times 8 \times 8 $
grid, with $ 6 \cdot 8 ^3 = 3072 $ cells. The space $ V _h ^1 $
consists of discontinuous piecewise linear vector fields, while
$ \widehat{ V } _h ^1 $ consists of cubic N\'ed\'elec edge elements of
the second kind, so that $ \ker \mathcal{B} _h \subset V _h ^1 $ are
linear N\'ed\'elec edge elements of the second kind, as described in
\autoref{sec:hybrid_nedelec} with $ r = 2 $.

Although the exact solution satisfies $ \operatorname{div} {D} = 0 $,
the numerical solution $ {D} _h $ drifts away from this constraint, as
measured by the $ H ( \operatorname{div} ; \mathcal{T} _h ) $ seminorm,
\begin{equation*}
  \lvert {D} _h \rvert _{ H ( \operatorname{div}; \mathcal{T} _h ) } \coloneqq \sqrt{  \sum _{ K \in \mathcal{T} _h } \lVert \operatorname{div} {D} _h \rVert ^2 _{ L ^2 (K) } } .
\end{equation*}
However, $ \operatorname{div} \widehat{ {D} } _h = 0 $ holds to
machine precision, as explained by
\autoref{thm:hybrid_maxwell_conservation}. Looking at $ {D} _h $
alone, one might think that this method fails to preserve the
charge-conservation constraint strongly. In fact, we have illustrated
that it actually \emph{does} preserve this constraint, when expressed
in terms of the numerical flux $ \widehat{ {D} } _h $ rather than
$ {D} _h $.

\begin{remark}
  The constraint behavior of $ {D} _h $ and $ \widehat{ {D} } _h $,
  observed in \autoref{fig:time}, is due to the finite element
  semidiscretization, not the time discretization.  Indeed, the
  charge-conservation constraint is linear, so if it holds for the
  semidiscretized system of ODEs, then any Runge--Kutta or partitioned
  Runge--Kutta method preserves it (\citet[Theorem
  IV.1.2]{HaLuWa2006}).
\end{remark}

\subsection{Frequency domain}

We next apply the hybrid approach to the frequency domain, again
assuming that $\epsilon$ and $\mu$ are positive constants with
$ \epsilon \mu = 1 $ and $ J = 0 $. This is done by first
approximating the Maxwell eigenvalue problem
\eqref{eqn:curlcurl_frequency} on $ \ker \mathcal{B} _h $ and then
applying hybrid post-processing, as follows:
\begin{itemize}
\item Find eigenpairs
  $ ( \omega _h ^2 , A _h ) \in \mathbb{R} ^{ + } \times \ker
  \mathcal{B} _h $ satisfying
  \begin{equation*}
    a ( A _h , A _h ^\prime ) = \omega _h ^2 \langle A _h , A _h ^\prime \rangle , \qquad \forall A _h ^\prime \in \ker \mathcal{B} _h ,
  \end{equation*}
  and let $ H _h \coloneqq \mu ^{-1} \operatorname{curl} A _h $ and
  $ {D} _h \coloneqq \epsilon ( - i \omega _h A _h ) $.

\item Find $ \widehat{ H } _h $ minimizing
  $ \lVert H _h - \widehat{ H } _h \rVert ^2 + \lVert i \omega _h {D}
  _h - \operatorname{curl} \widehat{ H } _h \rVert ^2 $ such that
  \begin{equation*}
    a ( A _h , A _h ^\prime ) + b ( A _h ^\prime , \widehat{ H } _h ) = \omega _h ^2 \langle A _h , A _h ^\prime \rangle , \qquad \forall A _h ^\prime \in V _h ^1 ,
  \end{equation*}
  and let
  $ \widehat{ {D} } _h \coloneqq - i \omega _h ^{-1}
  \operatorname{curl} \widehat{ H } _h $.
\end{itemize}
Note that this last step is equivalent to
$ i \omega _h \widehat{ {D} } _h = \operatorname{curl} \widehat{ H }
_h $, so $ \widehat{ H } _h $ can be seen as minimizing
$ \lVert H _h - \widehat{ H } _h \rVert ^2 + \omega _h ^2 \lVert {D}
_h - \widehat{ {D} } _h \rVert ^2 $.

\begin{figure}
  \centering
  \includegraphics{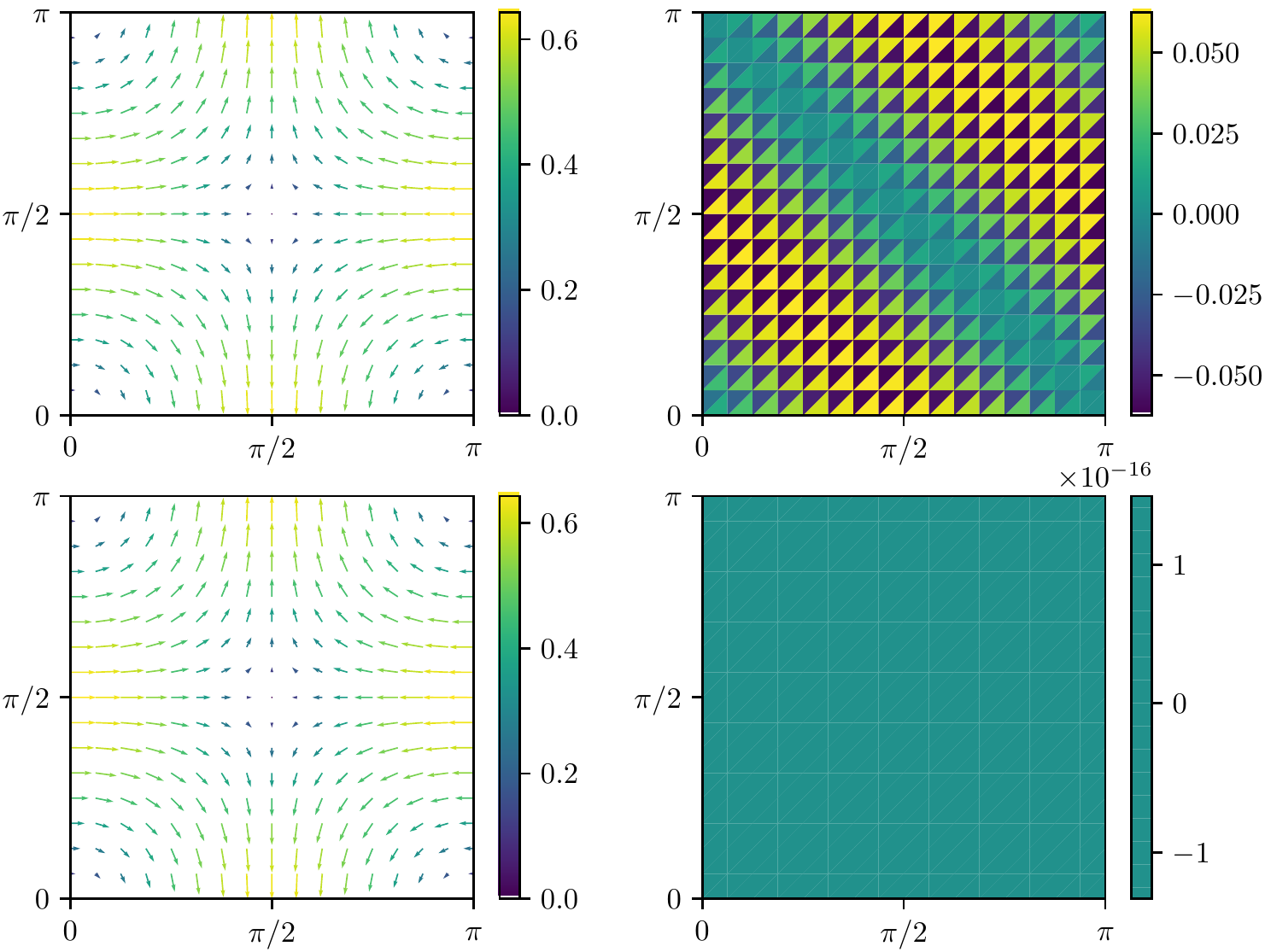}
  \caption{$ {D} _h $ and $ \operatorname{div} {D} _h $ (top row),
    compared to $ \widehat{ {D} } _h $ and
    $ \operatorname{div} \widehat{ {D} } _h $ (bottom row),
    approximating the Maxwell eigenmode with $\omega ^2 = 2$. While
    $ {D} _h $ and $ \widehat{ {D} } _h $ are nearly indistinguishable
    (left column), $ {D} _h $ fails to be strongly divergence-free,
    while $ \widehat{ {D} } _h $ is divergence-free to machine
    precision (right column).\label{fig:frequency}}
\end{figure}

We consider the 2-{D} square $ \Omega = ( 0, \pi ) ^2 $, where the
exact eigenvalues are sums of squares
($ \omega ^2 = 1, 1, 2, 4, 4, \ldots $). For simplicity, we look at
the approximation of the following analytical solution with simple
eigenvalue $ \omega ^2 = 2 $, assuming $ \epsilon = \mu = 1 $:
\begin{align*}
  A(x,y) &= \frac{ \sqrt{ 2 } } { \pi } ( -\cos x \sin y, \sin x \cos y ) ,\\
  H(x,y) &= \frac{ 2 \sqrt{ 2 } }{ \pi } \cos x \cos y ,\\
  {D} (x,y) &= \frac{ 2 i }{ \pi } ( \cos x \sin y, - \sin x \cos y ) .
\end{align*}
We take a uniform triangle mesh on an $ N \times N $ grid, which has
$ 2 N ^2 $ cells. As described in \autoref{sec:2D}, we take
$ V _h ^1 $ to consist of discontinuous piecewise degree-$(r-1)$
vector fields and $ \widehat{ V } _h ^0 $ to consist of
degree-$ (r+1) $ Lagrange elements, so that
$ \ker \mathcal{B} _h \subset V _h ^1 $ are degree-$(r-1)$
BDM edge elements.

\autoref{fig:frequency} shows $ {D} _h $ and $ \widehat{ {D} } _h $,
along with $ \operatorname{div} {D} _h $ and
$ \operatorname{div} \widehat{ {D} } _h $, for the case $ N = 16 $,
$ r = 2 $. Here, by $ \operatorname{div} {D} _h \in L ^2 (\Omega) $,
we mean the element-wise divergence
$ (\operatorname{div} {D} _h) \rvert _K \coloneqq \operatorname{div} (
{D} _h \rvert _K ) $ for each $ K \in \mathcal{T} _h $, since
$ {D} _h $ is in $ H ( \operatorname{div} ; \mathcal{T} _h ) $ but not
in $ H ( \operatorname{div} ; \Omega ) $. Although the vector fields
$ {D} _h $ and $ \widehat{ {D} } _h $ appear very similar, they behave
very differently with respect to the charge-conservation constraint:
$ \operatorname{div} {D} _h \neq 0 $, while
$ \operatorname{div} \widehat{ {D} } _h = 0 $ to machine
precision. Note that these are purely imaginary when $ A _h $ is real,
so the imaginary parts are plotted.

\begin{table}
  \centering
  \begin{tabular}{|l|c|c|c|c|c|c|c|c|c|}
\hline
\multirow{2}{*}{$r$} & mesh & 
\multicolumn{2}{|c|}{$\lVert H_h - H \rVert$} &
\multicolumn{2}{|c|}{$\lVert \widehat{H}_h - H \rVert$} &
\multicolumn{2}{|c|}{$\lVert D_h - D \rVert$} &
\multicolumn{2}{|c|}{$\lVert \widehat{D}_h - D \rVert$ \rule{0pt}{2.5ex}} \\
\cline{2-10}
& $N$ & error & rate & error & rate & error & rate & error & rate \\
\hline
\multirow{5}{*}{2}
 & 2 & 7.591e-01 & --- & 3.648e-01 & --- & 4.324e-01 & --- & 4.644e-01 & --- \\
 & 4 & 3.778e-01 & 1.007 & 1.070e-01 & 1.770 & 1.182e-01 & 1.872 & 1.342e-01 & 1.791 \\
 & 8 & 1.862e-01 & 1.021 & 2.753e-02 & 1.958 & 3.009e-02 & 1.974 & 3.512e-02 & 1.934 \\
 & 16 & 9.271e-02 & 1.006 & 6.926e-03 & 1.991 & 7.558e-03 & 1.993 & 8.906e-03 & 1.979 \\
 & 32 & 4.630e-02 & 1.002 & 1.734e-03 & 1.998 & 1.892e-03 & 1.998 & 2.236e-03 & 1.994 \\
\hline
\multirow{5}{*}{3}
 & 2 & 2.090e-01 & --- & 3.500e-02 & --- & 7.521e-02 & --- & 8.055e-02 & --- \\
 & 4 & 5.517e-02 & 1.922 & 2.750e-03 & 3.670 & 9.817e-03 & 2.938 & 9.960e-03 & 3.016 \\
 & 8 & 1.400e-02 & 1.978 & 1.827e-04 & 3.912 & 1.225e-03 & 3.002 & 1.220e-03 & 3.029 \\
 & 16 & 3.515e-03 & 1.994 & 1.159e-05 & 3.978 & 1.526e-04 & 3.005 & 1.512e-04 & 3.013 \\
 & 32 & 8.796e-04 & 1.999 & 7.270e-07 & 3.995 & 1.903e-05 & 3.003 & 1.882e-05 & 3.006 \\
\hline
\multirow{5}{*}{4}
 & 2 & 4.614e-02 & --- & 4.327e-03 & --- & 1.281e-02 & --- & 1.316e-02 & --- \\
 & 4 & 6.121e-03 & 2.914 & 1.250e-04 & 5.114 & 7.958e-04 & 4.008 & 8.629e-04 & 3.931 \\
 & 8 & 7.769e-04 & 2.978 & 3.759e-06 & 5.055 & 4.913e-05 & 4.018 & 5.500e-05 & 3.972 \\
 & 16 & 9.749e-05 & 2.994 & 1.155e-07 & 5.024 & 3.048e-06 & 4.011 & 3.454e-06 & 3.993 \\
 & 32 & 1.220e-05 & 2.999 & 3.582e-09 & 5.011 & 1.898e-07 & 4.006 & 2.160e-07 & 3.999 \\
\hline
\multirow{5}{*}{5}
 & 2 & 8.100e-03 & --- & 4.419e-04 & --- & 1.737e-03 & --- & 1.761e-03 & --- \\
 & 4 & 5.354e-04 & 3.919 & 6.307e-06 & 6.131 & 5.553e-05 & 4.968 & 5.321e-05 & 5.048 \\
 & 8 & 3.394e-05 & 3.980 & 9.434e-08 & 6.063 & 1.743e-06 & 4.993 & 1.642e-06 & 5.018 \\
 & 16 & 2.129e-06 & 3.995 & 1.447e-09 & 6.027 & 5.449e-08 & 5.000 & 5.105e-08 & 5.007 \\
 & 32 & 1.332e-07 & 3.999 & 2.404e-11 & 5.911 & 1.702e-09 & 5.000 & 1.592e-09 & 5.003 \\
\hline
\end{tabular}

  \vskip\baselineskip
  \caption{Convergence of the hybridized method for the
    $ \omega ^2 = 2 $ eigenmode of $ \Omega = ( 0, \pi ) ^2 $, using a
    uniform triangle mesh on an $ N \times N $ grid and degree-$(r-1)$
    BDM edge elements. The post-processed solution
    $ \widehat{ H } _h $ exhibits superconvergence relative to
    $ H _h $, while the errors and convergence rates of
    $ \widehat{ {D} } _h $ are comparable to those of
    $ {D} _h $.\label{tab:frequency}}
\end{table}

\autoref{tab:frequency} illustrates the convergence behavior of
$ H _h $, $ \widehat{ H } _h $, $ {D} _h $, and $ \widehat{ {D} } _h $
as the mesh parameter $ h \rightarrow 0 $, for elements of various
degrees. Since $ A _h $ is simply obtained by using degree-$(r-1)$ BDM
edge elements for the Maxwell eigenvalue problem, previous analyses of
this problem (e.g.,
\citep{Kikuchi1989,Hiptmair2002,Boffi2006,Boffi2007,ArFaWi2010} and
references therein) show that
$ \lVert A _h - A \rVert = \mathcal{O} ( h ^r ) $ and
$ \lVert \operatorname{curl} A _h - \operatorname{curl} A \rVert =
\mathcal{O} ( h ^{ r -1 } ) $, which imply the observed rates
$ \lVert {D} _h - {D} \rVert = \mathcal{O} ( h ^r ) $ and
$ \lVert H _h - H \rVert = \mathcal{O} ( h ^{ r -1 } )
$. Interestingly, for $ \widehat{ H } _h $ obtained by hybrid
post-processing, we observe the superconvergent rates
$ \lVert \widehat{ H } _h - H \rVert = \mathcal{O} ( h ^r ) $ for
$ r = 2 $ and $ \mathcal{O} ( h ^{ r + 1 } ) $ for $ r > 2 $. For
$ \widehat{ {D} } _h $, we observe errors comparable to those for
$ {D} _h $ and the same convergence rate,
$ \lVert \widehat{ {D} } _h - {D} \rVert = \mathcal{O} ( h ^r ) $.

We note that the observed rates of superconvergence, including the
reduced rate in the lowest-degree case, are the same as those obtained
for scalar elliptic problems by \citet*{BrDoMa1985} in the original
paper on the hybridized BDM method.

\section{Conclusion}

We have constructed a family of primal hybrid finite element methods
for Maxwell's equations, where the Lagrange multipliers enforcing
inter-element continuity and boundary conditions correspond to a
numerical trace $ \widehat{ H } _h $ of the magnetic field and a
numerical flux $ \widehat{ {D} } _h $ of the electric flux
density. These methods \emph{strongly} preserve the constraints
$ \operatorname{div} B _h = 0 $ and
$ \operatorname{div} \widehat{ {D} } _h = \rho $, the latter of which
corresponds to conservation of charge. As a special case, these
methods include hybridized versions of standard methods using
curl-conforming edge elements, which had previously been thought only
to be charge-conserving in a much weaker sense. We emphasize that
these conservative properties hold even if the methods are not
implemented in a hybrid fashion: if desired, $ \widehat{ H } _h $ and
$ \widehat{ {D} } _h $ may be recovered by an optional post-processing
step.

There are several natural directions for future work. First, the
numerical experiments in \autoref{sec:numerical} focused on hybridized
curl-conforming methods, due to the fact that their stability and
error analysis is already well established. However, as mentioned in
\autoref{sec:hybrid_nedelec}, this framework also includes
constraint-preserving nonconforming methods, which would be
interesting to investigate. Second, we do not yet have a complete
explanation of the hybrid superconvergence phenomenon for
$ \widehat{ H } _h \rightarrow H $; this is the subject of ongoing
work. Finally, the techniques developed here might be applied to study
constraint preservation in other families of hybrid methods,
particularly hybridizable discontinuous Galerkin (HDG) methods.

\subsection*{Acknowledgments}

Ari Stern acknowledges the support of the National Science Foundation
(DMS-1913272) and the Simons Foundation (\#279968). Yakov
Berchenko--Kogan was supported by an AMS--Simons Travel Grant.

\end{document}